\documentclass[11pt]{amsart}

\usepackage{bm}
\usepackage{fullpage}
\usepackage{amssymb}
\usepackage{amsmath,amsfonts,amsthm}
\usepackage{hyperref}
\usepackage{color}
\usepackage{cite}

\newtheorem{theorem}{Theorem}

\newtheorem{lemma}{Lemma}
\newtheorem{fact}{Fact}[theorem]
\newtheorem{definition}{Definition}
\newtheorem{claim}{Claim}
\newtheorem{proposition}{Proposition}

\usepackage{tikz}
\usepackage{graphicx}
\usepackage{subcaption}
\usetikzlibrary{decorations.markings}
\tikzstyle{vertex}=[circle, draw, inner sep=0pt, minimum size=6pt]

\title{ANOTHER VIEW OF BIPARTITE RAMSEY NUMBERS}
\author[1]{Yaser Rowshan$^1$}
\author[2]{MOSTAFA GHOLAMI$^1$}

\address{$^1$Department of Mathematics, Institute for Advanced Studies in Basic Sciences (IASBS), Zanjan 66731-45137, Iran}

\email{y.rowshan@iasbs.ac.ir,~~~y.rowshan.math@gmail.com}
\email{gholami.m@iasbs.ac.ir}

\begin{document}
	\maketitle
	
	\begin{abstract}
		For bipartite graphs $G$ and $H$ and a positive integer $m$, the $m$-bipartite Ramsey number $BR_m(G, H)$ of $G$ and $H$ is the smallest integer $n$,
		such that every red-blue coloring of $K_{m,n}$ results in a red $G$ or a blue $H$.
		Zhenming Bi, Gary Chartrand and Ping Zhang in \cite{bi2018another} evaluate this numbers for all positive integers $m$ when $G= K_{2,2}$  and
		$H \in \{K_{2,3}, K_{3,3}\}$, especially in a long and hard argument  they showed that  $BR_5(K_{2,2}, K_{3,3}) = BR_6(K_{2,2}, K_{3,3}) = 12$ and 
		$BR_7(K_{2,2}, K_{3,3}) = BR_8(K_{2,2}, K_{3,3}) = 9$. In this article, by a short and easy argument we determine the exact value of  $BR_m(K_{2,2}, K_{3,3})$ for each $m\geq 1$.
	\end{abstract}
	
	\section{Introduction}
	For given bipartite graphs $G_1,G_2,\ldots,G_t$ the  bipartite Ramsey number $BR(G_1,G_2,\ldots,G_t)$ is defined as the smallest positive integer $b$, such that any $t$-edge-coloring of the complete bipartite graph $K_{b,b}$ contains a monochromatic subgraph isomorphic to $G_i$, colored with the $i$th color for some $i$. One can refer to \cite{rowshan2021proof, gholami2021bipartite, hatala2021new,   rowshan2021size, chartrand2021new} and their references for further studies.
	
	We now consider red-blue colorings of complete bipartite graphs when the numbers of vertices in the two partite sets need not differ by at most $1$.  For bipartite graphs $G$ and $H$ and a positive integer $m$, the $m$-bipartite Ramsey number $BR_m(G, H)$ of $G$ and $H$ is the smallest integer $n$, such that every red-blue coloring of $K_{m,n}$ results in a red $G$ or a blue $H$. Zhenming Bi, Gary Chartrand and Ping Zhang in \cite{bi2018another} evaluate this numbers for all positive integers $m$ when $G= K_{2,2}$  and
	$H = K_{3,3}$, especially in a long and hard argument  they showed that:
	\begin{theorem}[Main results]\label{M.th}
		Suppose that  $m\geq 2 $ be a positive integer. Then:
		\[
		BR_m(K_{2,2},K_{3,3})= \left\lbrace
		\begin{array}{ll}	
			\text{does not exist}, & ~~~if~~~m=2,3,\vspace{.2 cm}\\
			15 & ~~~if~~~m=4,\vspace{.2 cm}\\
			12 & ~~~if~~~ m=5,6,\vspace{.2 cm}\\
			9 & ~~~if~~~ m=7,8.\vspace{.2 cm}\\
		\end{array}
		\right.
		\]	
	\end{theorem} 
	In this article, we come up with a short and easy argument  to prove Theorem \ref{M.th}.
	
	\section{Preparations}  
	In this article, we are only concerned with undirected, simple, and finite graphs. We follow \cite{bondy1976graph} for terminology and notations not defined here. Let $G$ be a graph with vertex set $V(G)$ and edge set $E(G)$. The degree of a vertex $v\in V(G)$ is denoted by $\deg_G(v)$, or simply by $\deg(v)$. The neighborhood $N_G(v)$ of a vertex $v$ is the set of all vertices of $G$ adjacent to $v$ and satisfies $|N_G(v)|=\deg_G(v)$. The minimum and maximum degrees of vertices of $G$ are denoted by $\delta(G)$ and $\Delta(G)$, respectively. 
	As usual, the complete bipartite graph with bipartition $(X,Y)$, where $|X|=m$ and $|Y|=n$, is denoted by $K_{m,n}$. We use $[X,Y]$ to denote the set of edges between a bipartition $(X,Y)$ of $G$. 
	The complement of a graph $G$, denoted by $\overline{G}$.  $H$ is $n$-colorable to $(G_1, G_2,\ldots, G_n)$ if there exists a $n$-edge decomposition of $H$, say $(H_1, H_2,\ldots, H_n)$ where $G_i\nsubseteq H_i$ for each $i=1,2, \ldots,n.$ We use $H\rightarrow (G_1, G_2,\ldots, G_n)$, to show that  $H$  is $n$-colorable to $(G_1, G_2,\ldots, G_n)$.
	\begin{definition}
		The Zarankiewicz number $z((m,n), K_{t,t})$ is defined as the maximum number of edges in any subgraph $G$ of the complete bipartite graph $K_{m,n}$, so that $G$ does not contain $K_{t,t}$ as a subgraph.
	\end{definition}
	By using the bounds in Table $4$ of \cite{collins2016zarankiewicz}, the following proposition  holds.
	
	\begin{proposition}\label{pro1}{(\cite{collins2016zarankiewicz})}The following result on Zarankiewicz number is true: 	
		\begin{itemize}
			\item[$\bullet$] $z((7,9), K_{3,3})\leq 40$.	
		\end{itemize}
	\end{proposition}
	Hattingh and Henning in \cite{hattingh1998bipartite} determined the exact value of the bipartite Ramsey number of $BR(K_{2,2}, K_{3,3})$ as follow:
	\begin{theorem}\label{th1}\cite{hattingh1998bipartite} $BR(K_{2,2}, K_{3,3})=9$.
	\end{theorem}

	\begin{lemma}\label{l1} Suppose that $G$ be a subgraph of $K_{m,n}$, where $n\geq6$, and $m\geq 4$. If there exists a vertex of $V(G)$ say $w$, so that $\deg_G(w)\geq 6$, then either $K_{2,2} \subseteq G$ or $K_{3,3} \subseteq \overline{G}$.
	\end{lemma}
	\begin{proof}
		W.l.g suppose that $w\in X$, and $N_G(w)=Y'$, where  $|Y'|\geq 6$. Assume that  $K_{2,2} \nsubseteq G$, therefore 	 $|N_G(w')\cap Y'|\leq 1$ fore each $w'\in X\setminus\{w\}$. Now, as $|X|\geq 4$ and $|Y'|=6$, then one can say that $K_{3,3} \subseteq \overline{G}[X\setminus\{w\}, Y']$, which means that the proof is complete.
	\end{proof}
	\section{\bf Proof of the main results}
	To prove our main results, namely Theorem \ref{M.th}, we begin with  the following theorem.	 
	\begin{theorem}\label{t1}
		$BR_4(K_{2,2}, K_{3,3})=15$.
	\end{theorem}
	\begin{proof}
		\medskip
		 By Figure \ref{fi0}, one can  check that $K_{2,2} \nsubseteq G$.  Also, by  Figure \ref{fi0}, it can be said that for each $X'=\{x,x',x''\}\subseteq \{x_1,\ldots,x_2\}$ and $Y'=\{y,y',y''\}\subseteq \{y_1,\ldots,y_{14}\}$,  there is at least one edge of $E([X',Y'])$ say $e$, so that $e\in E(G)$. Which means that $\overline{G}$ is $K_{3,3}$-free. So,	$K_{4,14} \rightarrow (K_{2,2},K_{3,3})$.
		
		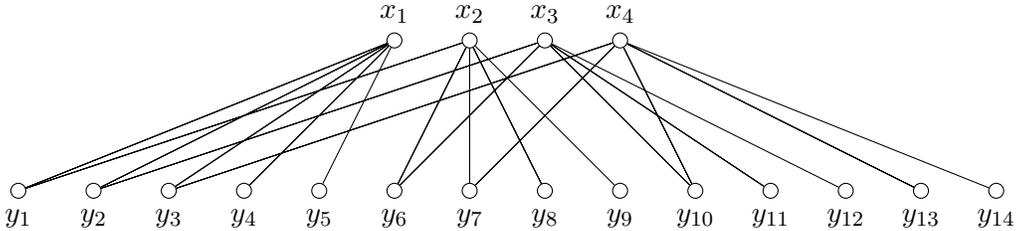
\begin{figure}[ht]
			\begin{tabular}{ccc}
				\begin{tikzpicture}
					\node [draw, circle, fill=white, inner sep=2pt, label=below:$y_1$] (y1) at (0,0) {};
					\node [draw, circle, fill=white, inner sep=2pt, label=below:$y_2$] (y2) at (1,0) {};
					\node [draw, circle, fill=white, inner sep=2pt, label=below:$y_3$] (y3) at (2,0) {};
					\node [draw, circle, fill=white, inner sep=2pt, label=below:$y_4$] (y4) at (3,0) {};
					\node [draw, circle, fill=white, inner sep=2pt, label=below:$y_5$] (y5) at (4,0) {};
					\node [draw, circle, fill=white, inner sep=2pt, label=below:$y_6$] (y6) at (5,0) {};
					\node [draw, circle, fill=white, inner sep=2pt, label=below:$y_7$] (y7) at (6,0) {};
					\node [draw, circle, fill=white, inner sep=2pt, label=below:$y_8$] (y8) at (7,0) {};
					\node [draw, circle, fill=white, inner sep=2pt, label=below:$y_9$] (y9) at (8,0) {};
					\node [draw, circle, fill=white, inner sep=2pt, label=below:$y_{10}$] (y10) at (9,0) {};
					\node [draw, circle, fill=white, inner sep=2pt, label=below:$y_{11}$] (y11) at (10,0) {};
					\node [draw, circle, fill=white, inner sep=2pt, label=below:$y_{12}$] (y12) at (11,0) {};
					\node [draw, circle, fill=white, inner sep=2pt, label=below:$y_{13}$] (y13) at (12,0) {};
					\node [draw, circle, fill=white, inner sep=2pt, label=below:$y_{14}$] (y14) at (13,0) {};
					\
					\node [draw, circle, fill=white, inner sep=2pt, label=above:$x_1$] (x1) at (5,2) {};
					\node [draw, circle, fill=white, inner sep=2pt, label=above:$x_2$] (x2) at (6,2) {};
					\node [draw, circle, fill=white, inner sep=2pt, label=above:$x_3$] (x3) at (7,2) {};
					\node [draw, circle, fill=white, inner sep=2pt, label=above:$x_4$] (x4) at (8,2) {};
					
					\draw (x1)--(y1)--(x1)--(y2)--(x1)--(y3)--(x1)--(y4)--(x1)--(y5);
					\draw (x2)--(y1)--(x2)--(y6)--(x2)--(y7)--(x2)--(y8)--(x2)--(y9);
					\draw (x3)--(y2)--(x3)--(y6)--(x3)--(y10)--(x3)--(y11)--(x3)--(y12);
					\draw (x4)--(y3)--(x4)--(y7)--(x4)--(y10)--(x4)--(y13)--(x4)--(y14);
				\end{tikzpicture}

			\end{tabular}\\
			\caption{Edge disjoint subgraphs $G$ and $\overline{G}$ of $K_{4,14}$.}
			\label{fi0}
		\end{figure}
		Let $(X=\{x_1,x_2,x_{3},x_4\},Y=\{y_1,y_2,\ldots ,y_{15}\})$ be the partition  sets of  $K_{4,15}$ and suppose that $G\subseteq K_{4,15}$, such that $K_{2,2} \nsubseteq G$. Consider $\Delta=\Delta (G_X)$( the maximum degree of  vertices in the part $X$ in $G$). Since $K_{2,2} \nsubseteq G$, if $\Delta\geq 6$ then the proof is complete by Lemma \ref{l1}. Also, if $\Delta\leq 4$, then  $K_{3,3}\subseteq \overline{G}$. Hence assume that $\Delta=5$ and let $\Delta=|Y'= N_G(x)|$. Since $K_{2,2} \nsubseteq G$, thus $|N_G(x')\cap Y'|\leq 1$ for each $x'\neq x$. If either there exists a vertex of $X\setminus \{x\}$ say $x'$, so that $|N_G(x')\cap Y_1|=0$ or  there exist two vertices of $X\setminus \{x\}$ say $x', x''$, such that $N_G(x')\cap Y'=N_G(x'')\cap Y'$, then it can be said that $K_{3,3}\subseteq \overline{G}[X\setminus \{x\},Y']$. Therefore, we may assume that $|N_G(x')\cap Y'|=1$ and  $N_G(x')\cap Y'\neq N_G(x'')\cap Y'$ for each $ x',x''\in X\setminus \{x\}$. W.l.g let $x=x_1$ and $Y'=\{y_1,\ldots,y_5\}$. Now, for $i=2,3$,  consider $|N_G(x_i)\cap( Y\setminus Y_1)|$. As $\Delta=5$, then either  $|N_G(x_i)\cap (Y\setminus Y_1)|\leq 3$ for one $i=2,3$, or $|N_G(x_i)\cap(Y\setminus Y_1)|=4$ and $|N_G(x_2)\cap N_G(x_3)\cap (Y\setminus Y_1)|=1$. Therefore, as $|Y|=15$, it is easy to say that $|\cup_{i=1}^{i=3}N(x_i)\cap Y|\leq 12$, that is $K_{3,3}\subseteq\overline{G}[\{x_1,x_{2},x_{3}\}, Y\setminus \cup_{i=1}^{i=3}N(x_i)]$, which means that the proof is complete.
	\end{proof}
	\begin{theorem}\label{t2} Suppose that $m\in \{5,6\}$, then  $BR_m(K_{2,2}, K_{3,3})=12$.
	\end{theorem}
	\begin{proof} If we prove the theorem for $m=5$, then for $m=6$ the proof is trivial. Hence, let $m=5$ and assume that $(X=\{x_1,x_2,x_{3},x_4,x_5\},Y=\{y_1,y_2,\ldots ,y_{12}\})$ be the partition  sets of  $K_{5,12}$ and $G\subseteq K_{5,12}$, where $K_{2,2} \nsubseteq G$. Consider $\Delta (G_X)=\Delta$. Since $K_{2,2} \nsubseteq G$, if $\Delta\geq 6$ then the proof is complete by Lemma \ref{l1}. For $\Delta\leq 3$, it is clear that $K_{3,3} \subseteq \overline{G}$. Hence $\Delta\in \{4,5\}$. 
		
		First assume that $\Delta=4$. W.l.g assume that  $Y_1=\{y_1,\ldots,y_4\}=N_G(x_1)$. Since $K_{2,2} \nsubseteq G$, we have $|N_G(x_i)\cap Y_1 |\leq 1$, for each $x_i\in X\setminus\{x_1\}$. Now we have the following claims:
		\begin{claim}\label{c1}
			For each $x\in X$, we have $|N_G(x)|=4=\Delta$.
		\end{claim}	
		\begin{proof}[Proof of Claim \ref{c1}]
			By contrary assume that $|N_G(x)|\leq 3$ for at least one member of $ X\setminus\{x_1\}$ say $x'$. Let  $N_G(x')=Y_2$. As $|X|=5$, there are at least two vertices of $X\setminus\{x'\}$ say $x_i,x_j$, such that $|N_G(x_i)\cap Y_2|= 1$, otherwise  $K_{3,3}\subseteq\overline{G}[X\setminus \{x_2\}, Y_2]$. Therefore as $|Y|=12$ and $\Delta =4$ it can be said that  $K_{3,3}\subseteq\overline{G}[\{x',x_i,x_j\}, Y\setminus Y_2]$.	 
		\end{proof}
		
		\begin{claim}\label{c2}
			For each $x\in X\setminus\{x_1\}$, we have $|N_G(x_1)\cap N_G(x) |=1$.
		\end{claim}	
		\begin{proof}[Proof of Claim \ref{c2}]
			By contradiction, let $|N_G(x_1)\cap N_G(x) |=0$ for at least one member of $X\setminus\{x_1\}$ say $x$. W.l.g let $x=x_2$ and by Claim \ref{c1} let  $N_G(x_2)\cap Y= \{y_5,y_6,y_7,y_8\}$. For $i=1,2$, as $|Y_i|=4$, if either $|N_G(x_j)\cap Y_i|= 0$ for at least one $i\in \{1,2\}$ and one $j\in \{3,4,5\}$ or there exist $j,j'\in \{3,5,5\}$ such that $|N_G(x_j)\cap N_G(x_{j'})\cap Y_i |=1$ for one $i\in\{1,2\}$, then $K_{3,3}\subseteq\overline{G}[X\setminus \{x_i\}, Y_i]$. So, let $|N_G(x_j)\cap Y_i |= 1$ and $N_G(x_j)\cap Y_i\neq N_G(x_{j'})\cap Y_i$ for each $i\in\{1,2\}$ and each $j,j'\in \{3,4,5\}$. W.l.g let  $x_3y_1, x_3y_5, x_4y_2, x_4y_6, x_5y_3, x_5y_7\in E(G)$. Now, since $|Y|=12$, by Claim \ref{c1}, we have $|N_G(x_j)\cap Y_3 |=2$ for each $x\in \{x_3,x_4,x_5\}$, in which $Y_3=\{y_9,y_{10}, y_{11}, y_{12}\}$, Therefore one can say that there exists at least one vertex of $Y_3$ say $y$, so that $|N_G(y)\cap \{x_3,x_4,x_5\}|\leq 1$. W.l.g let $y=y_9$ and assume that $x_3y_9, x_4y_9\in E(\overline{G})$. Hence one can say that $K_{3,3}\subseteq\overline{G}[\{x_2,x_3,x_4\}, \{y_3,y_4,y_9\}]$.
		\end{proof}
		So by Claim \ref{c1}, w.l.g let   $Y_1=\{y_1,\ldots,y_4\}=N_G(x_1)$, and  $Y_2=\{y_1,y_5,y_6,y_7\}=N_G(x_2)$ and by  Claim \ref{c2}, $|N_G(x_i)\cap N_G(x) |=1$ for each $i=1,2$ and each $x\neq x_i$. Now, as $\Delta=4$ and $|Y|=12$,  if there exists a vertex of $\{x_3,x_4,x_5\}$ say $x$, so that $xy_1\in E(\overline{G})$, then it can be checked that  $K_{3,3}\subseteq\overline{G}[\{x_1,x_2,x\}, Y\setminus (Y_1\cup Y_2)]$. Hence assume that $x_iy_1\in E(G)$ for each $i=2,3,4,5$, which means that $K_{3,3}\subseteq\overline{G}[\{x_2,x_3,x_4\}, Y_1\setminus \{y_1\}]$. Hence for the case that $\Delta=4$ the theorem holds.
		
		Now assume that $\Delta=5$. Let $X'=\{x\in X, ~\deg_G(x)=5\}$. Now we have the following fact:
		\begin{fact}\label{f1}
			If	$|X'|= 1$, then the proof is complete.
		\end{fact}
		{\bf Proof of the fact:} We may suppose that  $ X'=\{x_1\}$. W.l.g let $Y_1=\{y_1,\ldots,y_5\}=N_G(x_1)$. Since $K_{2,2} \nsubseteq G$, we have $|N_G(x_i)\cap Y_1 |\leq 1$, for each $x_i\in X\setminus\{x_1\}$.  As $|Y_1|=5$, one can assume that $|N_G(x_i)\cap Y_1 |= 1$ and $N_G(x_i)\cap Y_1\neq N_G(x_j)\cap Y_1$ for each $i,j\in \{2,3,4,5\}$, otherwise $K_{3,3}\subseteq\overline{G}[ X\setminus\{x_1\}, Y_1]$. Therefore, 
		w.l.g suppose that  $x_2y_1, x_3y_2, x_4y_3, x_5y_4\in E(G)$. Now, one can say that $|N_G(x_i)\cap (Y\setminus Y_1 )|=3$ for at least three vertices of $X\setminus \{x_1\}$. Otherwise,  if there exist at least two vertices of $X\setminus \{x_1\}$ say $x_2,x_3$ so that  $|N_G(x_i)\cap (Y\setminus Y_1 )|\leq 2$, since $|Y|=12$ it can be said that $K_{3,3}\subseteq\overline{G}[\{x_1,x_2,x_3\}, Y\setminus Y_1]$. So, assume that $|N_G(x_i)\cap (Y\setminus Y_1 )|=3$ for each $i\in \{2,3,4\}$ and let $Y_i= N_G(x_i)$.  Now, w.l.g we may suppose that $Y_2=\{y_1, y_6,y_7,y_8\}=N_G(x_2)$.  As $K_{2,2} \nsubseteq G$, we have $|N_G(x_i)\cap (\{y_6,y_7,y_8\})|\leq 1$ for each $i\in \{3,4,5\}$.
		With symmetry, for each $i\in \{6,7,8\}$, we have $|N_G(y_i)\cap \{x_3,x_4,x_5\} |\geq 1$.  Otherwise,  if there exists at least one vertex of $Y_2\setminus \{y_1\}$ say $y$ so that  $|N_G(y)\cap (\{x_3,x_4,x_5\} )|=0$, then   $K_{3,3}\subseteq\overline{G}[\{x_3,x_4,x_5\}, \{y_1,y_5,y\}]$.  Hence w.l.g let $x_3y_6, x_4y_7, x_5y_8\in E(G)$ and suppose that $Y_3=\{y_1, y_6,y_9,y_{10}\}=N_G(x_3)$.  As $K_{2,2} \nsubseteq G$, we have $N_G(y_9)\cap (\{x_4,x_5\}) \neq N_G(y_{10})\cap (\{x_4,x_5\}) $, and $|N_G(y_9)\cap \{x_4,x_5\} |=|N_G(y_{10})\cap \{x_4,x_5\} |=  1$. W.l.g let $x_4y_9, x_5y_{10}\in E(G)$. Since $|N_G(x_4)\cap (Y\setminus Y_1 )|=3$, we have $|N_G(x_4)\cap \{y_{11}, y_{12}\}|=1$, w.l.g assume that $x_4y_{11} \in E(\overline{G})$. Therefore,  $K_{3,3}\subseteq\overline{G}[\{x_2, x_3, x_4\}, \{y_4, y_5, y_{12}\}]$, which means that the proof of the fact is complete. 
		
		So, by Fact \ref{f1}, assume that  $x_1,x_2\in X'$, and $Y_1=\{y_1,\ldots,y_5\}=N_G(x_1)$. Since $K_{2,2} \nsubseteq G$, we have $|N_G(x_i)\cap Y_1 |\leq 1$, for each $x_i\in X\setminus\{x_1\}$. Also since $|Y_1|=5$, one can assume that $|N_G(x_i)\cap Y_1 |= 1$ and $N_G(x_i)\cap Y_1\neq N_G(x_j)\cap Y_1$ for each $i,j\in \{2,3,4,5\}$, otherwise $K_{3,3}\subset \overline{G}$. Therefore,  w.l.g suppose that $Y_2=\{y_1, y_6,y_7,y_8, y_9\}$ and $x_3y_2, x_4y_3, x_5y_4\in E(G)$. With symmetry we have $|N_G(x_i)\cap Y_2\setminus\{y_1\} |= 1$ and $N_G(x_i)\cap (Y_2\setminus\{y_1\}) \neq N_G(x_j)\cap (Y_2\setminus\{y_1\}) $ for each $i,j\in \{3,4,5\}$. Now, w.l.g we may suppose that $x_3y_6, x_4y_7, x_5y_8\in E(G)$. Therefore one can say that $K_{3,3}\subseteq\overline{G}[\{x_3, x_4, x_5\}, \{y_1, y_5, y_9\}]$.  Hence, $BR_m(K_{2,2}, K_{3,3})\leq 12$ for $m=5,6$.  
		
		To show that $BR_m(K_{2,2}, K_{3,3})\geq 12$, decompose the edges of $K_{6,11}$ into graphs $G$ and $\overline{G}$, where $G$ is shown in Figure \ref{fi2}. By Figure \ref{fi2} it can be checked that,  $K_{6,11} \rightarrow (K_{2,2},K_{3,3})$, which means that $BR_m(K_{2,2}, K_{3,3})= 12$.
		
		\begin{figure}[ht]
			\begin{tabular}{ccc}
				\begin{tikzpicture}
					\node [draw, circle, fill=white, inner sep=2pt, label=below:$y_1$] (y1) at (0,0) {};
					\node [draw, circle, fill=white, inner sep=2pt, label=below:$y_2$] (y2) at (1,0) {};
					\node [draw, circle, fill=white, inner sep=2pt, label=below:$y_3$] (y3) at (2,0) {};
					\node [draw, circle, fill=white, inner sep=2pt, label=below:$y_4$] (y4) at (3,0) {};
					\node [draw, circle, fill=white, inner sep=2pt, label=below:$y_5$] (y5) at (4,0) {};
					\node [draw, circle, fill=white, inner sep=2pt, label=below:$y_6$] (y6) at (5,0) {};
					\node [draw, circle, fill=white, inner sep=2pt, label=below:$y_7$] (y7) at (6,0) {};
					\node [draw, circle, fill=white, inner sep=2pt, label=below:$y_8$] (y8) at (7,0) {};
					\node [draw, circle, fill=white, inner sep=2pt, label=below:$y_9$] (y9) at (8,0) {};
					\node [draw, circle, fill=white, inner sep=2pt, label=below:$y_{10}$] (y10) at (9,0) {};
					\node [draw, circle, fill=white, inner sep=2pt, label=below:$y_{11}$] (y11) at (10,0) {};

					\
					\node [draw, circle, fill=white, inner sep=2pt, label=above:$x_1$] (x1) at (3,2) {};
					\node [draw, circle, fill=white, inner sep=2pt, label=above:$x_2$] (x2) at (4,2) {};
					\node [draw, circle, fill=white, inner sep=2pt, label=above:$x_3$] (x3) at (5,2) {};
					\node [draw, circle, fill=white, inner sep=2pt, label=above:$x_4$] (x4) at (6,2) {};
					\node [draw, circle, fill=white, inner sep=2pt, label=above:$x_5$] (x5) at (7,2) {};
					\node [draw, circle, fill=white, inner sep=2pt, label=above:$x_6$] (x6) at (8,2) {};
					
					\draw (x1)--(y1)--(x1)--(y2)--(x1)--(y3)--(x1)--(y4);
					\draw (x2)--(y1)--(x2)--(y5)--(x2)--(y6)--(x2)--(y7);
					\draw (x3)--(y1)--(x3)--(y8)--(x3)--(y9)--(x3)--(y10);
					\draw (x4)--(y2)--(x4)--(y5)--(x4)--(y8)--(x4)--(y11);
					\draw (x5)--(y3)--(x5)--(y6)--(x5)--(y9)--(x5)--(y11);
					\draw (x6)--(y4)--(x6)--(y7)--(x6)--(y10)--(x6)--(y11);
				\end{tikzpicture}

			\end{tabular}\\
			\caption{Edge disjoint subgraphs $G$ and $\overline{G}$ of $K_{6,11}$.}
			\label{fi2}
		\end{figure}
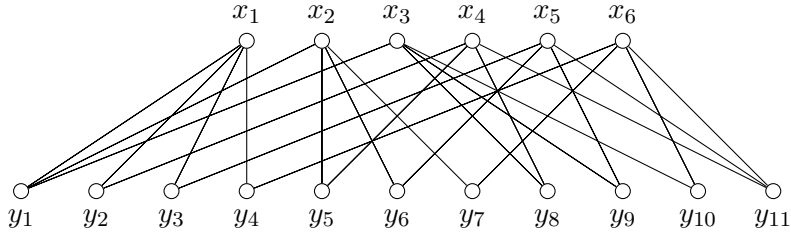
	\end{proof}
	
	\begin{theorem}\label{t3} Suppose that $m\in \{7,8\}$, then  $BR_m(K_{2,2}, K_{3,3})=9$.
	\end{theorem}
	\begin{proof}
		Suppose that $(X=\{x_1,\ldots,x_7\},Y=\{y_1,y_2,\ldots ,y_{9}\})$ be the partition  sets of  $K_{7,9}$.
		Consider $G\subseteq K_{7,9}$, where $K_{2,2} \nsubseteq G$. If there exists a vertex of $X$ say $x$, such that $\deg_G(x)\geq 5$, then as   $K_{2,2} \nsubseteq G$, we have $|N_G(x_i)\cap N_G(x) |\leq 1$, hence $|X|=7$, by the pigeon-hole principle it can be said that $K_{3,3}\subseteq \overline{G}[X\setminus\{x\}, N_G(x)]$. Also by Proposition \ref{pro1}, since $z((7,9),K_{3,3})\leq 40$, one can say that $|N_G(x)|=4$  for at least two vertices of $X$. Otherwise  $|E(\overline{G})|\geq 41$, so $K_{3,3}\subseteq\overline{G}$. W.l.g let $\deg_G(x_i)=4$ for each $x\in X'$, and let $x_1,x_2\in X'$.  Now, we have the following claim:
		\begin{claim}\label{c3}
			For each  $x\in X$ and each $x'\in X'$, $|N_G(x)\cap N_G(x') |=1$.
		\end{claim}	
		\begin{proof}[Proof of Claim \ref{c3}]
			By contradiction, let $|N_G(x)\cap N_G(x') |=0$ for some $x\in X$ and some $x'\in X'$. If $|N_G(x)\cap N_G(x') |=0$ for at least two vertices of $X$, then it is clear that $K_{3,3}\subseteq \overline{G}[X, Y_1]$. So, w.l.g let $x'=x_1$ and  $N_G(x_1)= Y_1$. Therefore as $|Y_1|=4$ and $|X|=7$, then by the pigeon-hole principle  there exist at least two vertices of $X\setminus\{x_1,x\}$ say $x',x''$ so that $N_G(x')\cap Y_1 =N_G(x'')\cap Y_1$, which means that $K_{3,3}\subseteq \overline{G}[\{x,x', x''\}, Y_1]$. 
		\end{proof}
		Now, by Claim \ref{c3} w.l.g let $N_G(x_1)\cap Y_1= \{y_1,y_2,y_3,y_4\}$ and $N_G(x_2)\cap Y_2= \{y_1,y_5,y_6,y_7\}$. By considering $|X'|$ we have two case as follow:
		
		{\bf Case 1: $|X'|\geq 3$.} W.l.g assume that $x_3\in X'$, therefore Claim \ref{c3} limits us to   $N_G(x_3)\cap Y_3= \{y,y',y_8,y_9\}$, where $y\in \{y_2,y_3, y_4\}$ and $y'\in \{y_5,y_6,y_7\}$. W.l.g assume that $y=y_2, y'=y_5$. So, as $K_{2,2}\nsubseteq G$ we have $|X'|=3$, and $|N_G(x)\cap Y_i|\leq 1$ for each $i=1,2,3$ and each $x\in X\setminus X'$. If there exist at least two vertices of $X\setminus X'$ say $x',x''$ such  that $|N_G(w)\cap \{y_1, y_2, y_5\}|=1$, for each $w\in \{x',x''\}$, then $|N_G(w)|\leq 2$, otherwise  $K_{2,2}\subseteq G$, a contradiction, so as $|X'|=3$ and $|N_G(x')|\leq 2$, we have $|E(G)|\leq 22$, therefore $|E(\overline{G})|\geq 41$, and by Proposition \ref{pro1}, $K_{3,3}\subseteq\overline{G}$. Hence, suppose that $|N_G(x')\cap \{y_1, y_2, y_5\}|=0$ for at least three vertices of $ X\setminus X'$.  Which means that $K_{3,3}\subseteq \overline{G}[X\setminus X', \{y_1, y_2, y_5\}]$. 
		
		{\bf Case 2: $|X'|= 2$.}  By Proposition \ref{pro1}, we have $|N_G(x)|=3$ for each $x\in X\setminus X'$. Now we have the following claim:
		
		\begin{claim}\label{c4}
			If there exist a vertex of $X\setminus X'$ say $x$, so that  $xy_1\in E(G)$, then $K_{3,3}\subseteq \overline{G}$.
		\end{claim}	
		\begin{proof}[Proof of Claim \ref{c4}]
			If $xy_1\in E(G)$  for at least two vertices of $X\setminus X'$, then $K_{3,3}\subseteq \overline{G}[X\setminus \{x_1\}, Y_1\setminus \{y_1\}]$.	So, w.l.g let $x_3y_1\in E(G)$. Since   $|N_G(x_3)|=3$, we have $N_G(x_3)=Y_3=\{y_1,y_8,y_9\}$. Now  consider $X''=\{x_4,x_5,x_6,x_7\}$. As $|N_G(x)|=3$, for each $x\in X''$, we have $|N_G(x)\cap Y_i\setminus \{y_1\}|=1$. Also as $|X''|=4$, and  $|Y_i\setminus \{y_1\}|=3$ for $i=1,2$, then  by the pigeon-hole principle  there are at least two vertices of $X''$, say $x_4, x_5$, such that $N_G(x_4)\cap Y_1\setminus \{y_1\}=N_G(x_5)\cap Y_1\setminus \{y_1\}=\{y\}$. W.l.g assume that $y=y_2$. Also by the pigeon-hole principle one can say that there is at least one vertex of  $Y_2\setminus \{y_1\}$ say $y'$, such that $y'x_4, y'x_5\in E(\overline{G})$. W.l.g let $y'=y_5$. Hence, it can be checked that $K_{3,3}\subseteq \overline{G}[\{x_3,x_4, x_5\}, \{y_3,y_4,y_5\}]$. 
		\end{proof}
		Now, by Claim \ref{c4} we may assume that 	$xy_1\in E(\overline{G})$  for each $x\in X\setminus X'$. 	By Claim \ref{c3},	for each  $x\in X\setminus \{x_1,x_2\}=X''$ and each $x'\in \{x_1,x_2\}$, we have $|N_G(x)\cap N_G(x') |=1$.  Now, since $|X\setminus X'|=5$, by the pigeon-hole principle  there exists  two vertices of $\{y_2,y_3,y_4\}$ say $y_2, y_3$,  such that $|N_G(y_i)\cap X'''|=2$, where $i=2,3$. W.l.g let $N_G(y_2)\cap X'''=\{x_3,x_4\}$ and $N_G(y_3)\cap X'''=\{x_5,x_6\}$. Since $K_{2,2}\nsubseteq G$, $xy_1\in E(\overline{G})$ and  $|N_G(x)\cap N_G(x_2) |=1$, we may assume that $x_3y_5, x_4y_6\in E(G)$. Also for at least one $i\in \{5,6\}$ w have $x_iy_7\in E(\overline{G})$, otherwise it can be said that $K_{2,2}\subseteq G$, a contradiction. Hence assume that $x_5y_7\in E(\overline{G})$, Therefore one can check that $K_{3,3}\subseteq \overline{G}[\{x_3,x_4, x_5\}, \{y_1,y_4,y_7\}]$. Which means that in any case $K_{3,3}\subseteq \overline{G}$.
		
		Hence by Cases 1,2, we have  $BR_7(K_{2,2}, K_{3,3})\leq 9$.  This also implies that every red-blue coloring of $K_{8,9}$ results in a red $K_{2,2}$ or a blue $K_{3,3}$. Therefore by Theorem \ref{th1} as $K_{8,8}\rightarrow (K_{2,2}, K_{3,3})$ we have $BR_m(K_{2,2}, K_{3,3})= 9$ where $m=7,8$. Hence the proof is complete.
	\end{proof}
\begin{proof}[\bf Proof of Theorem \ref{M.th}] For $m=2,3$, it is easy to say that $BR_m(K_{2,2}, K_{3,3})$ does not exist. Now,
	by combining Theorem \ref{t1}, \ref{t2} and \ref{t3}  we conclude that the proof of Theorem \ref{M.th} is complete.
\end{proof}
	\bibliographystyle{plain}
	\bibliography{BI}
	
\end{document}